\RequirePackage{fix-cm}
\documentclass{svjour3}             
\smartqed  
\usepackage{setspace}

\usepackage{amsmath}
\usepackage{graphicx}
\usepackage{hyperref}
\usepackage{mathrsfs}
\usepackage{epstopdf}
\usepackage{amsfonts}
\usepackage{booktabs}
\usepackage[top=1in, bottom=1in, left=1in, right=1in]{geometry}
\usepackage{bm}
\usepackage{color}
\newcommand{\bpsi}{\bm{\psi}}
\newcommand{\bphi}{\bm{\phi}}
\newcommand{\bp}{{\bm{p}}}

\newcommand{\be}{{\bm{e}}}

\begin{document}

\title{ On First Come, First Served Queues with Two Classes of Impatient Customers 
}


\author{Ivo Adan         \and
        Brett Hathaway \and
        Vidyadhar G. Kulkarni
}


\institute{Ivo Adan
			\at Department of Industrial Engineering, Eindhoven University of Technology, 5600 MB Eindhoven, the Netherlands, \\ Tel.: +31 (40) 2472932 \\
              \email{iadan@.tue.nl}
                            \and Brett Hathaway
              \at Kenan-Flagler School of Business, University of North Carolina at Chapel Hill, Chapel Hill, NC 27599, \\
              Tel.: 801-837-0474 \\
              \email{brett\textunderscore hathaway@kenan-flagler.unc.edu}
              \and
              V. G. Kulkarni
            \at Department of Statistics and Operations Research, University of North Carolina at Chapel Hill, Chapel Hill, NC 27599, \\ Tel.: 919-962-3837 \\
              \email{vkulkarn@email.unc.edu}
}

\date{Received: date / Accepted: date}

\maketitle

\begin{abstract}
We study systems with two classes of impatient customers who differ across the classes in their distribution of service times and patience times. The customers are served on a first-come, first served basis (FCFS) regardless of their class. Such systems are common in customer call centers, which often segment their arrivals into classes of callers whose requests differ in their complexity and criticality. We first consider an $M$/$G$/1 + $M$ queue and then analyze the $M$/$M$/$k$ + $M$ case. Analyzing these systems using a queue length process proves intractable as it would require us to keep track of the class of each customer at each position in queue. Consequently, we introduce a virtual waiting time process where the service times of customers who will eventually abandon the system are not considered. We analyze this process to obtain performance measures such as the percentage of customers receiving service in each class, the expected waiting times of customers in each class, and the average number of customers waiting in queue. We use our characterization to perform a numerical analysis of the $M$/$M$/$k$ + $M$ system, and find several managerial implications of administering a FCFS system with multiple classes of impatient customers. Finally, we compare the performance a system based on data from a call center with the steady-state performance measures of a comparable $M$/$M$/$k$ + $M$ system. We find that the performance measures of the $M$/$M$/$k$ + $M$ system serve as good approximations of the system based on real data.
\PACS{PACS code1 \and PACS code2 \and more}
\end{abstract}

\section{Introduction}
In this paper we analyze queueing systems with different classes of customers who may abandon (renege from) the system if their waiting time exceeds their patience time, i.e., the maximum amount of time they are willing to wait before abandoning the system. This work is motivated primarily by customer call centers, which often segment their callers into different classes. Since call centers are the prevalent customer-facing service channel of many organizations, they often receive a variety of caller requests which may differ significantly in their service requirements and criticality. For example, banking call centers receive requests as simple as balance inquiries and as complex and critical as dealing with fraudulent activity on a caller's account. While a service representative can obtain an account balance relatively quickly, handling fraudulent activity takes longer as it involves a higher level of legal expertise and paperwork. Furthermore, because fraudulent activity is usually more critical than obtaining a balance, callers who are calling about fraud may be more patient than callers who are calling to obtain a balance. Because callers' service requests vary so greatly, call centers often train subsets of their representatives to handle only certain types of service requests. Based on the service the callers request from the phone menu, the Automatic Call Distributor (ACD) segments callers into classes and routes the callers within each class to the appropriately trained subset of representatives. Depending on which types of requests are included in each class, these classes may differ from each other with respect to their distribution of service times and patience times. Consequently, one subset of representatives may serve a queue that receives arrivals from multiple classes that differ from each other in their typical service requirements and their callers' patience levels. Call centers sometimes give priority to certain classes based on criteria such as the callers' value to the organization or the criticality of the callers' service request. However, in an effort to be fair, call centers often serve their callers on a first-come, first-served basis (FCFS) regardless of class. Because the FCFS policy is such a common practice, it is important to describe the performance of such systems. In this study we do this by characterizing the performance of FCFS systems with two customer classes that may differ from each other in both their distribution of service times and their distribution of patience times.

As queue abandonment is a common customer behavior in many service systems, it is not surprising that a large number of studies have been devoted to characterizing queueing systems with impatient customers.  Approaches for describing the performance of systems with a single class of impatient customers have included analytical characterizations (Daley \cite{daley1965general}, Baccelli and Hebuterne \cite{baccelli1981queues}, Baccelli et al. \cite{baccelli1984single}, Stanford \cite{stanford1979reneging}), and performance approximations (Garnett et al. \cite{garnett2002designing}, Zeltyn and Mandelbaum \cite{zeltyn2005call}, Iravani and Balc{\i}og̃lu \cite{iravani2008approximations}). The literature also contains several studies of two-class systems. In most of these studies, one customer class is prioritized over the other. Choi et al. \cite{choi2001m} analyze the underlying Markov process of an $M$/$M$/1 queue where one class of customers with constant patience times receives preemptive priority over a second class of customers who have no impatience. The authors obtain the joint distribution of the system size, and the Laplace transform (LT) of the response time of the second class of customers. Brandt and Brandt \cite{brandt2004two} extend the approach in \cite{choi2001m} to include generally distributed patience times for the first class of customers. Iravani and Balc{\i}og̃lu \cite{iravani2008priority} use the level-crossing technique proposed in Brill and Posner \cite{brill1977level} and Brill and Posner \cite{brill1981system} to study two $M$/$GI$/1 settings. They first consider a preemptive-resume discipline where customers in both classes have exponentially distributed patience times, and then consider a non-preemptive discipline where customers in the first class have exponentially distributed patience times, but customers in the second class have no impatience. Iravani and Balc{\i}og̃lu obtain the waiting time distributions for each class and the probability that customers in each class will abandon.

A handful of papers have dealt with priority queues with two classes of impatient customers in a multi-server setting. Motivated by call centers where callers may leave a voicemail, Brandt and Brandt \cite{brandt1999two} consider a multi-server system where callers from the first class are impatient and receive priority over callers from the second class, who have no impatience. Callers from the first class who renege may join the second class by leaving a voicemail and are only contacted when the number of idle servers in the system exceeds some threshold. The authors obtain the exact distribution of the number of callers in service from the first class, and approximations of the moments of the number of callers in service from the second class. Jouini and Roubos \cite{jouini2014multiple} consider an $M$/$M$/$s$ + $M$ queueing system where all customers have the same mean service times and mean patience times, but callers in one of the classes receive non-preemptive priority. Within each class, customers may be served according to a FCFS or last-come, first-served (LCFS) discipline. Jouini and Roubos obtain the mean unconditional waiting times, the mean waiting times conditional on receiving service, and the mean waiting times conditional on abandoning the system for both classes under several policy permutations.

The studies that are most pertinent to our work are those in which classes of impatient customers are served on a FCFS basis, regardless of their class. Gurvich and Whitt \cite{gurvich2007service} approximate the performance of a multiclass call center with impatient callers under the quality-and-efficiency-driven (QED) regime introduced by Halfin and Whitt \cite{halfin1981heavy}. They analyze how the call center performs under a class of asymptotically optimal routing policies, of which the FCFS policy is a special case. Talreja and Whitt \cite{talreja2008fluid} rely on a deterministic fluid model to approximate the performance of a multiserver, multiclass FCFS system with impatient customers in an overloaded, efficiency-driven (ED) regime. Adan et al. \cite{adan2013design} design heuristics to determine the staffing levels  required to meet target service levels in an overloaded FCFS multiclass system with impatient customers. Van Houdt \cite{van2012analysis} considers an $MAP$/$PH$/1 multiclass queue where customers in each of the classes have a general distribution of patience times. Van Houdt develops a numerical procedure for analyzing the performance characteristics of the system by reducing the joint workload and arrival processes into a fluid queue, and expresses the steady state measures using matrix analytical methods. His method produces an exact characterization of the waiting time distribution and abandonment probability under a discrete distribution of patience times, and approximations of the same performance measures under a continuous distribution of patience times. Sakuma and Takine \cite{sakuma2017multi} study the $M$/$PH$/1 system and assume that customers within each class have the same deterministic patience time. In addition to obtaining the waiting time distribution and abandonment probabilities, they obtain the joint queue-length distribution. Finally, Sarhangian and Balc{\i}og̃lu \cite{sarhangian2013waiting} study two multiclass FCFS systems. The first system is an $M$/$G$/1+$M$ queue where customers across classes may differ in their service time distribution, but all customers share common exponentially distributed patience times. The second system is an $M$/$M$/$c$+$M$ queue where customers across classes may differ in their exponentially distributed patience times, but all customers share the same exponentially distributed service times. For both systems, the authors obtain the LT of the virtual waiting time for each of the $k$ classes by exploiting the level crossing technique in \cite{brill1977level} and \cite{brill1981system}. They then relate the virtual waiting time to the actual waiting time to compute steady-state performance measures such as the mean waiting times and the percentage of customers who renege from each class.

In this study we analyze two systems in which two classes of impatient customers are served according to an FCFS policy. The distinction between our setting and the settings in previous studies is that in our setting customers across classes may differ in their distributions of their service times and patience times, while in previous settings customers across classes may only differ in one of the two distributions. This distinction is crucial in characterizing the performance of multiclass systems with customers whose service times and patience may vary based on the complexity and criticality of their requests. We first consider an $M$/$G$/1+$M$ queue, and then analyze an $M$/$M$/$k$+$M$ queue. To characterize the performance of these systems, we introduce a virtual waiting time process as described in Benes \cite{benes1963general} and Tak{\'a}cs, et al. \cite{takacs1962introduction} (see Heyman and Sobel \cite{heyman1982stochastic}, pp. 383-390 for details). In a virtual waiting time process the service times of customers who will eventually abandon the system are not considered. By analyzing this process we obtain performance characteristics such as the percentage of customers who receive service in each class, the expected waiting times of customers in each class, and the average number of customers waiting in queue from each class. Note that although a related formula for the virtual waiting time in a single class $M$/$G$/1+$M$ queue is reported in \cite{brandt2013workload}, it is not suitable for direct computation as it consists of an exponentially growing number of terms. We next perform a numerical analysis of the $M$/$M$/$k$+$M$ system under various arrival rates, mean service times, and mean patience times. Our analysis demonstrates that accounting for differences across classes in the distribution of customers' service times and patience times is critical as the performance of our system differs considerably from a system where only the service time distribution varies across classes. The results of our numerical analysis have broad managerial implications including service level forecasting, revenue management, and the evaluation of server productivity. As a final exercise, we compare the simulated performance of a system based on data from a multiclass call center with the performance measures of a comparable $M$/$M$/$k$+$M$ system. To construct our simulated system, we select two classes from the data that differ in their distribution of service times and caller patience times. We find that the performance measures from the $M$/$M$/$k$+$M$ system serve as good approximations of the performance of the simulated system based on the call center data.

The remainder of the paper is organized as follows.  In Sect. \ref{sec:M/G/1} we analyze the $M$/$G$/1+$M$ queue.  In Sect. \ref{sec:M/M/k} we analyze the $M$/$M$/$k$+$M$ queue, including a special case where the two classes share a common mean service time. In Sect. \ref{sec:pm} we derive steady-state performance measures. In Sect. \ref{sec:Numerical} we present our numerical analysis, and in Sect. \ref{sec:Approx} we compare the performance of the simulated system based on real data and our analytical characterization of the $M$/$M$/$k$+$M$ system.

\section{M/G/1+M System}\label{sec:M/G/1}
We begin with a single server queueing system serving two classes of customers, see Figure \ref{fig:model1}. Assume that class $i$ ($i=1,2$) customers arrive according to independent Poisson processes with rate $\lambda_i$ and
need independent and identically distributed (iid) service times with cdf $G_i(\cdot)$ and mean $\tau_i$. The customers are impatient, and customers from class $i$ leave the system after an exponential amount of time with parameter $\theta_i$ (called the patience time) if their queueing time is longer than their patience time. The patience times of customers are independent of each other. For this system we are interested in performance characteristics such as the long-run fraction of customers entering service, server utilization, and the expected waiting time for service. Note that due to the impatience of the customers in each class, the system will always be stable even if the total arrival rate exceeds the service rate.

\begin{figure}[htb]
\begin{center}
\includegraphics[width=0.3\linewidth]{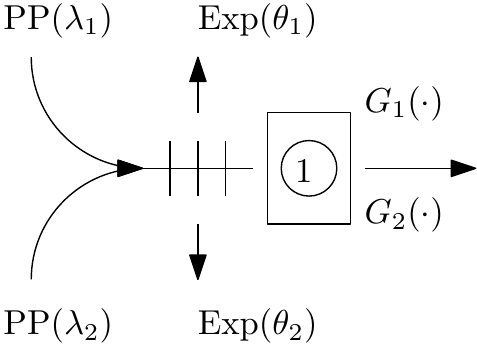}
\caption{Single server model.}\label{fig:model1}
\end{center}
\end{figure}

It seems natural to study this model through its queue length process. However, keeping track of the number of customers from each class who are in the queue is not sufficient. We would also need to keep track of the class of each customer at each position in queue, since patience times depend on customer class. This renders the Markov process intractable. Thus, we introduce the virtual queueing time process below. This process appears to be tractable.

Let $W(t)$ be the virtual queueing time at time $t$.  We know that $W(t)$ decreases with rate 1 at all times while it is positive. If an arrival from class $i$ occurs at time $t$ and $W(t) = w$, the arrival leaves without service with probability $1-e^{-\theta_i w}$, or enters service with probability $e^{-\theta_i w}$, and needs a random amount of service with cdf $G_i(\cdot)$. Hence the distribution of $S_i(t)$, the size of the upward jump at time $t$ due to a class $i$ arrival, given $W(t) = w$, is given by
\begin{equation} \label{eq:cdfst}
 P(S_i(t) \le y|W(t) = w) = 1-e^{-\theta_i w} + e^{-\theta_i w}G_i(y)
 \end{equation}
and thus
\begin{eqnarray*}
  E(e^{-s S_i(t)}|W(t)=w) &=& 1-e^{-\theta_i w} + e^{-\theta_i w} \tilde{G}_i(s)  ,
\end{eqnarray*}
where $\tilde{G}_i(\cdot)$ is the LST of $G_i(\cdot)$.
Let
\begin{eqnarray*}
  \psi(s,t) &=& E(e^{-s W(t)}), \\
  p_0(t) &=& P(W(t)=0), \\
  \phi(s,t) &=& \psi(s,t) - p_0 (t)
\end{eqnarray*}
and
\begin{eqnarray*}
  \psi(s) &=& \lim_{t \rightarrow \infty} \psi(s,t)\\
  & = & E(e^{-sW}) ,\label{eq:psidef}\\
  p_0 &=& \lim_{t \rightarrow \infty} p_0(t) \\
  &=&  P(W=0), \nonumber\\
  \phi(s) &=& \psi(s) - p_0, \nonumber
\end{eqnarray*}
 where $W$ is the limit (in distribution) of $W(t)$ as $t \rightarrow \infty$. In the interval $(t,t+h],$ an arrival of type $i$ occurs with probability $\lambda_ih + o(h)$, and no event occurs with probability $1 - (\lambda_1 + \lambda_2)h + o(h)$ as $h \rightarrow 0$. Then we get 
\begin{eqnarray*}
\psi(s,t+h) & = &
(1-\lambda_1 h-\lambda_2 h) \phi (s,t) e^{sh} + (1-\lambda_1 h-\lambda_2 h) p_0(t) \\
& & + \sum_{i=1}^2 \lambda_i h (\phi (s,t) - \phi(s+\theta_i,t) + \psi ( s+\theta_i,t) \tilde{G}_i(s) ) + o(h).
\end{eqnarray*}
Inserting $e^{sh} = 1 + sh + o(h)$, rearranging terms and dividing by $h$ and then letting $h \rightarrow 0$, we obtain
\begin{eqnarray*}
 \frac{d}{dt}\psi (s,t) &= & s\phi (s,t) - \sum_{i=1}^2 \lambda_i \psi(s+\theta_i,t) (1-  \tilde{G}_i(s)).
\end{eqnarray*}
Now let $t \rightarrow \infty$. Then $\frac{d}{dt}\psi(s,t) \rightarrow 0$, $\psi(s,t) \rightarrow \psi(s)$ and
$\phi (s,t) \rightarrow \phi(s)$, so
\begin{eqnarray*}
0 &= & s\phi (s) - \sum_{i=1}^2 \lambda_i \psi(s+\theta_i) (1-\tilde{G}_i(s)) .
\end{eqnarray*}
Dividing by $s$ and using
\[ {H}_i(s) = \lambda_i  \frac{1-\tilde{G}_i(s)}{ s}, \]
we finally get
\begin{equation}\label{eq:psip0}
\psi(s) = p_0 + \sum_{i=1}^2  \psi(s+\theta_i) {H}_i(s).
\end{equation}
Note that ${H}_i(s)/(\lambda_i \tau_i)$ is the LST of the equilibrium distribution of the service times of customers from class $i$.

Sarhangian and Balc{\i}og̃lu \cite{sarhangian2013waiting} have derived this equation by a different method, however they could solve it only in the special case $\theta_1 = \theta_2$.  Here we develop an efficient  method to solve Eq.~\ref{eq:psip0} even when $\theta_1 \ne \theta_2$. We have included the detailed derivation of this equation, since the same steps can be used to derive the equations in the multi-server set up studied in the next section.

Repeated application of Eq.~\eqref{eq:psip0} shows that its solution can be written as
\begin{equation}\label{eq:psis}
\psi(s) = p_0 c(s),
\end{equation}
where
\[
c(s) = \sum_{i=0}^\infty\sum_{j=0}^\infty c_{i,j}(s).
\]
The terms $c_{i,j}(s)$ satisfy the recursion
\[ c_{i,j}(s) = {H}_1(s+(i-1)\theta_1+j\theta_2)c_{i-1,j}(s) + {H}_2(s+ i\theta_1+(j-1)\theta_2)c_{i,j-1}(s),\]
with initially $c_{0,0} (s)= 1$ and $c_{i,j} (s)= 0$ if $i<0$ or $j<0$.
The recursive procedure to obtain \eqref{eq:psis} as well as convergence properties of the series $c(s)$ will be explained in more detail in Sect.~\ref{sec:M/M/k}, where we analyze the multi server system.
Finally, to determine $p_0$, we use $\psi(0)=1$ in \eqref{eq:psip0}, yielding
\[
 p_0  = 1 - \sum_{i=1}^2 \psi(\theta_i)\lambda_i \tau_i = 1 - p_0 \sum_{i=1}^2 c(\theta_i) \lambda_i \tau_i ,
\]
so
\[
p_0 = \left[ 1+\sum_{i=1}^2 c(\theta_i) \lambda_i \tau_i \right]^{-1}.
\]
We shall see in Sect.~\ref{sec:pm} that many performance measures can be computed in terms of $\psi(\theta_i)$.


\begin{remark}(Hyper-Exponential impatience) \label{rm:hyp}
{\rm
For Hyper-Exp($\theta_{ij}, p_{ij}$) impatience of class $i$ customers,
it is straightforward to derive the  equation
\[
\psi(s) = p_0 + \sum_{i=1}^2 \sum_j p_{ij} \psi(s+\theta_{ij}) {H}_i (s),
\]
where
\[
p_0 = 1-\sum_{i=1}^2 \sum_j p_{ij} \psi(\theta_{ij}) \lambda_i \tau_i.
\]
}
\end{remark}

One can also show that our results reduce to known results (see Daley \cite{daley1965general} for example) when specialized to a system with a single class of customers.

%
%
%
%
\section{M/M/k+M System}\label{sec:M/M/k}
\subsection{Model}
Now we consider a FCFS multi server system with $k$ servers serving two classes of customers, see Fig. \ref{fig:model2}. However, unlike in Sect. ~\ref{sec:M/G/1}, we now assume that the service times are exponentially distributed. To be precise, we assume that customers from class $i$ arrive according to a PP($\lambda_i$), need iid Exp($\mu_i$) service times, and exhibit iid Exp($\theta_i$) patience times ($i=1,2$). If service does not start before the patience time expires, the customer leaves without service.

\begin{figure}[htb]
\begin{center}
\includegraphics[width=0.3\linewidth]{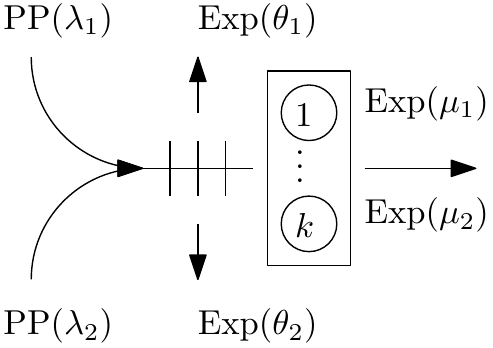}
\caption{Multi server model.}\label{fig:model2}
\end{center}
\end{figure}

\subsection{Virtual queueing time process}

As in the single server case, we study the virtual queueing time process, augmented by a supplementary variable to construct a Markov process. Let $W(t)$ be the virtual queueing time at time $t$ in this system. This is the queueing time that would be experienced by a virtual customer arriving at time $t$. Let $N_i(t)$ be the number of servers serving a class $i$ customer {just after} time $t+W(t)$ but before the next customer (if there is one) entering service at time $t+W(t)$. This means that $N_i(t)$ is the number of servers 
busy with a class $i$ customer just before a customer arriving at time $t$ enters service. So $N_1 (t) + N_2 (t)$ is always at most $k-1$.

This unusual definition enables us to determine the size of the upward jump of the virtual queueing time if an arriving customer at time $t$ decides to join the queue (since his patience exceeds $W(t)$). The jump is the minimum of the service time of the arriving customer and the residual service times of the customers in service at the moment he enters service at time $t+W(t)$. As we will explain below, $\{(W(t),N_1(t),N_2(t)), t \ge 0\}$ is a Markov process with upward jumps, the size of which depend on $W(t)$, and a continuous downward deterministic drift of rate 1 between jumps.

Suppose $W(t)=0$ and $(N_1(t),N_2(t)) = (i,j)$.
Then $(i,j)$ is the number of busy servers of class $1$ and $2$ at time $t$.
For $0 \le i+j \le k-1$ the transition rates of services in state $(0,i,j)$ are given by
\begin{eqnarray*}
q_{(0,i,j),(0,i-1,j)} & = & i\mu_1, \\
q_{(0,i,j),(0,i,j-1)} & = & j\mu_2
\end{eqnarray*}
and for $0 \le i+j < k-1$ the transition rates of arrivals are
\begin{eqnarray*}
q_{(0,i,j),(0,i+1,j)} & = & \lambda_1, \\
q_{(0,i,j),(0,i,j+1)} & = & \lambda_2.
\end{eqnarray*}

This accounts for all transitions from states $(0,i,j), \; 0 \le i+j \le k-1$, except for the transition rates of arrivals in states with $i+j=k-1$. These transition rates are described below.

Now suppose the state of the tri-variate process at time $t$ is $(w,i,j)$ with $w \ge 0$ and $i+j=k-1$.
This means just after time $t+w$ we will have $i$ busy servers of class 1 and $j$ servers of class 2. Consider an arrival from class 1 at time $t$. This customer has to wait an amount $w$ for service to begin. He reneges before his service starts with probability $1 - e^{-\theta_1 w}$, in which case the state does not change, and he enters service at time $t+w$ with probability $e^{-\theta_1 w}$. Then the next departure occurs after an {Exp}$((i+1)\mu_1 + j\mu_2)$ time $X$ and the departure at time $t+w+X$ is from class 1 with probability $(i+1) \mu_1 /((i+1) \mu_1 + j \mu_2)$ and from class 2 with probability $j \mu_2 /((i+1) \mu_1 + j \mu_2 )$. In the first case the state jumps at time $t$ from $(w,i,j)$ to $(w+X,i,j)$. In the second case the state jumps to $(w+X,i+1,j-1)$. The process is similar in case of a class 2 arrival. Hence we get the following transition rates from states $(w,i,j)$ with $w \ge 0$ and $i+j=k-1$:
\begin{eqnarray*}
q_{(w,i,j),([w+x,w+x+dx),i+1,j-1)} & = & \lambda_1  \times e^{-\theta_1 w}((i+1) \mu_1 + j \mu_2 ) e^{-((i+1)\mu_1+j\mu_2)x}dx  \times \frac{j \mu_2}{(i+1) \mu_1  + j \mu_2 }\\
& = & \lambda_1  e^{-\theta_1 w} j \mu_2 e^{-((i+1)\mu_1+j\mu_2)x}dx ,\\
q_{(w,i,j),([w+x,w+x+dx),i,j)} & = & \lambda_1 e^{-\theta_1 w}  (i+1) \mu_1 e^{-((i+1) \mu_1 +j \mu_2)x}dx + \lambda_2  e^{-\theta_2 w}  (j+1) \mu_2  e^{-(i\mu_1+(j+1)\mu_2)x}dx ,\\
q_{(w,i,j),([w+x,w+x+dx),i-1,j+1)} & = & \lambda_2 e^{-\theta_2 w} i \mu_1 e^{-(i\mu_1+(j+1)\mu_2)x}dx .\\
\end{eqnarray*}
Between upward jumps, the $W$ process decreases continuously and deterministically at rate 1 while it is positive.
When $W$ reaches 0 in state $(0,i,j)$, the process will stay in this state until either an arrival or service completion occurs.

\begin{figure}[htb]
\begin{center}
\includegraphics[width=.85\linewidth]{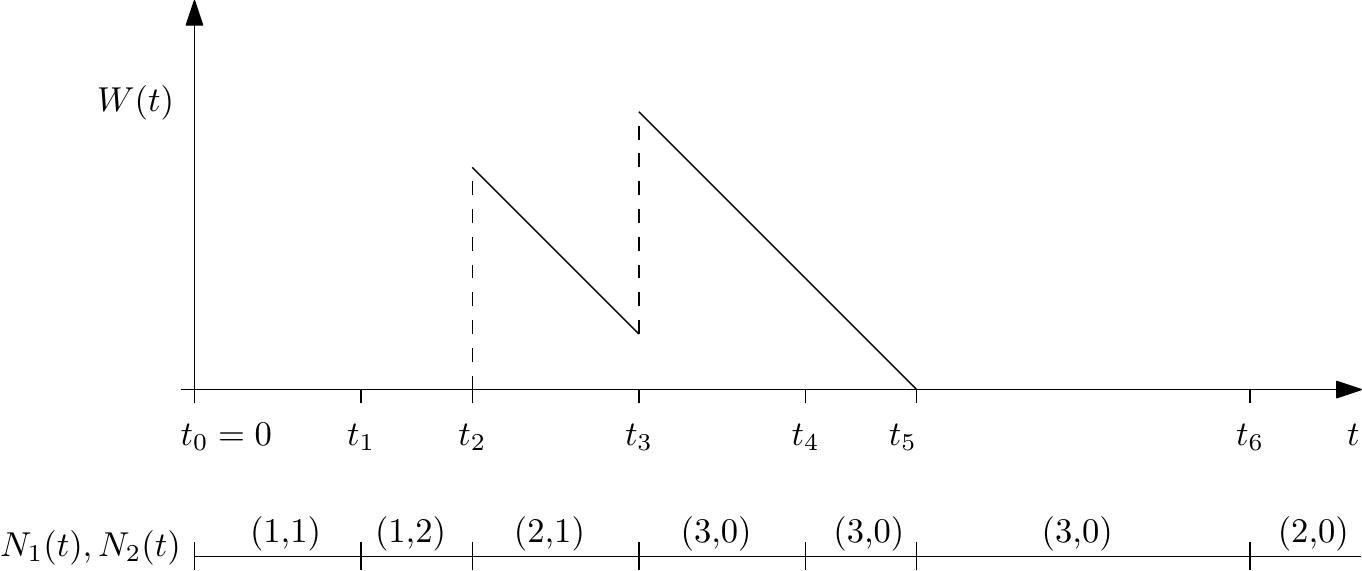}
\caption{Sample path of the virtual queueing time process of a system with 4 servers.
}\label{fig:sample}
\end{center}
\end{figure}

A sample path of the Markov process $\{(W(t),N_1(t),N_2(t)), t \ge 0\}$ is shown in Fig. \ref{fig:sample}. It describes the following events.  At time $t_0 = 0$ the system is in state $(0,1,1)$, which means that the virtual queueing time is 0, 2 servers are busy, one with a class 1 customer and the other with a class 2 customer. Then a class 2 customer arrives at time $t_1$, so the system states changes to state $(0,1,2)$. At time $t_2$ a class 1 customer arrives and  $W(t)$ jumps an
Exp($2\mu_1+2\mu_2$) amount. At time $t_2+W(t_2)$ a class 2 customer will complete service (from the two class 1 and two class 2 customers in service), so $(N_1(t),N_2(t))$ jumps to $(2,1)$ at time $t_2$. The next customer arrives at time $t_3$. He is from class 1 and his patience exceeds $W(t_3)$. So this customer joins the queue and $W(t)$ jumps an Exp($3\mu_1+\mu_2$) amount. At time $t_3+W(t_3)$ the class 2 customer will complete service before one of the three class 1 customers in service, so $(N_1(t),N_2(t))$ jumps to $(3,0)$ at time $t_3$. At time $t_4$ a class 2 customer arrives, but his patience is less than $W(t_4)$, and thus he leaves the system without receiving service. At time $t_5 = t_3+W(t_3)$ the class 2 customer completes service and the virtual queueing time reaches 0. The last event in Fig. \ref{fig:sample} occurs at time $t_6$. A class 1 customer completes service and the system state jumps to $(0,2,0)$.

We next determine the Laplace-Stieltjes transform (LST) of the virtual queueing time.

\subsection{Steady-state analysis}

We first introduce the notation
\begin{eqnarray*}
\psi_i(s,t)  &=& E(s^{-sW(t)};N_1(t) = i, N_2(t) = k-1-i), \quad 0 \le i \le k-1,\\
p_{i,j} (t) &=& P(W(t) = 0,   N_1(t) = i, N_2(t) = j), \quad 0 \le i+j \le k-1,\\
\phi_i(s,t) & = & \psi_i (s,t) - p_{i,k-1-i}(t), \quad  0 \le i \le k-1
\end{eqnarray*}
and
\begin{eqnarray*}
\psi_i(s) & = & \lim_{t \rightarrow \infty} \psi_i(s,t)\\
& = & E(s^{-sW};N_1 = i, N_2 = k-1-i), \quad 0 \le i \le k-1, \\
p_{i,j} & = & \lim_{t \rightarrow \infty} p_{i,j} (t) \\
& = & P(W=0, N_1 = i, N_2 = j), \quad 0 \le i + j \le k-1,\\
\phi_i(s) & = & \psi_i (s) - p_{i,k-1-i} ,
\end{eqnarray*}
where $(W,N_1,N_2)$ is the limit (in distribution) of $(W(t),N_1(t),N_2(t))$ as $t \rightarrow \infty$.

We start with the balance equations for the steady-state probabilities $p_{i,j}$. Let
\[
\bp_n = [p_{0,n}\;\;p_{1,n-1}\;\;\cdots\;\;p_{n,0}] , \quad 0 \le n \le k-1.
\]
Then the balance equations can be written in vector-matrix form as
\begin{eqnarray}
\bp_0 (\lambda_1+\lambda_2) &=& \bp_1 M_1, \nonumber \\
\bp_n (\lambda_1+\lambda_2) + \bp_n \Delta_{n} &=& \bp_{n-1} \Lambda_{n-1} + \bp_{n+1} M_{n+1}, \quad 1 \le n < k-1, \label{eq:b1}
\end{eqnarray}
where the $(n+1) \times (n+1)$ matrix $\Delta_{\mu,n} = {\rm diag}(n\mu_2, \mu_1 + (n-1)\mu_2, \ldots, n\mu_1)$ and the non-zero elements of the $(n+1) \times (n+2)$ matrix $\Lambda_n = [\lambda_{n,i,j}]$ and $(n+1) \times n$ matrix $M_n = [\mu_{n,i,j}]$ are given by
\begin{eqnarray*}
\lambda_{n,i,i+1} & = & \lambda_1, \quad 0 \le i \le n,\\
\lambda_{n,i,i} & = & \lambda_2, \quad 0 \le i \le n,\\
 \mu_{n,i,i-1} & = & i\mu_1, \quad 1 \le i \le n,\\
 \mu_{n,i,i} & = & (n-i)\mu_2, \quad 0 \le i \le n-1.
\end{eqnarray*}
The above equations can be simplified to
\begin{equation}\label{eq:bb1}
\bp_{n} = \bp_{n+1} R_{n+1} , \quad 0 \le n < k-1.
\end{equation}
The $(n+1) \times n $ matrices $R_n$ in \eqref{eq:bb1} recursively follow from
\begin{eqnarray*}
R_1     &=& M_1 (\lambda_1+\lambda_2)^{-1}, \\
R_{n+1} &=& M_{n+1} ( (\lambda_1+\lambda_2)I + \Delta_n - R_n \Lambda_{n-1} )^{-1}, \quad 1 \le n < k-1 ,
\end{eqnarray*}
where $I$ denotes the identity matrix.

Now we proceed to derive differential equations for the time-dependent LSTs $\psi_i(s,t)$ and then take $t$ to infinity to obtain the steady-state equations. For small $h > 0$ we get
\begin{eqnarray*}
  \psi_i (s,t+h) &=&  (1-\lambda_1 h-\lambda_2 h) \phi_i (s,t) e^{sh}
  + (1-\lambda_1 h-\lambda_2 h
  -(i\mu_1 + (k-1-i)\mu_2)h) p_{i,k-1-i} (t)
  \\
  & & + \lambda_1 h (\phi_i (s,t) - \phi_i(s+\theta_1,t)) + \lambda_2 h (\phi_i (s,t) - \phi_i(s+\theta_2,t)) \\
  & & + \lambda_1 h \psi_i ( s+\theta_1,t) \frac{(i+1)\mu_1}{s+(i+1)\mu_1 + (k-1-i)\mu_2} \\
  & & + \lambda_1 h \psi_{i-1} ( s+\theta_1,t) \frac{(k-i)\mu_2}{s+i\mu_1 + (k-i)\mu_2} \\
  & & + \lambda_2 h \psi_i ( s+\theta_2,t) \frac{(k-i)\mu_2}{s+i\mu_1 + (k-i)\mu_2} \\
  & & + \lambda_2 h \psi_{i+1} ( s+\theta_2,t) \frac{(i+1)\mu_1}{s+(i+1)\mu_1 + (k-1-i)\mu_2} \\
  & &
  + \lambda_1 h p_{i-1,k-1-i} (t) + \lambda_2 h p_{i,k-2-i} (t)+ o(h),
\end{eqnarray*}
where, by convention, $p_{i,j} (t) = 0$ if $i < 0$ or $j <0$.
Inserting $e^{sh} = 1 + sh + o(h)$, rearranging terms and dividing by $h$ and then letting $h \rightarrow 0$, we obtain
\begin{eqnarray*}
 \frac{d}{dt}\psi_i (s,t) &= & s\phi_i (s,t) - \lambda_1 \psi_i(s+\theta_1,t) - \lambda_2 \psi_i(s+\theta_2,t)
  -(i\mu_1 + (k-1-i)\mu_2) p_{i,k-1-i} (t)  \\
  & & + \lambda_1 \psi_i ( s+\theta_1,t) \frac{(i+1)\mu_1}{s+(i+1)\mu_1 + (k-1-i)\mu_2} \\
  & & + \lambda_1 \psi_{i-1} ( s+\theta_1,t) \frac{(k-i)\mu_2}{s+i\mu_1 + (k-i)\mu_2} \\
  & & + \lambda_2 \psi_i ( s+\theta_2,t) \frac{(k-i)\mu_2}{s+i\mu_1 + (k-i)\mu_2} \\
  & & + \lambda_2 \psi_{i+1} ( s+\theta_2,t) \frac{(i+1)\mu_1}{s+(i+1)\mu_1 + (k-1-i)\mu_2} \\
  & & + \lambda_1 p_{i-1,k-1-i} (t) + \lambda_2 p_{i,k-2-i} (t).
\end{eqnarray*}
Now let $t \rightarrow \infty$. Then the system reaches steady state and $\frac{d}{dt}\psi_i(s,t) \rightarrow 0$, $\psi_i(s,t) \rightarrow \psi_i(s)$,
$\phi_i (s,t) \rightarrow \phi_i(s)$ and $p_{i,j} (t) \rightarrow p_{i,j}$, so
\begin{eqnarray}\label{eq:psii}
0 &= & s\phi_i (s) - \lambda_1 \psi_i(s+\theta_1) - \lambda_2 \psi_i(s+\theta_2) - (i\mu_1 + (k-1-i)\mu_2) p_{i,k-1-i}
\nonumber \\
  & & + \lambda_1 \psi_i ( s+\theta_1) \frac{(i+1)\mu_1}{s+(i+1)\mu_1 + (k-1-i)\mu_2} \nonumber \\
  & & + \lambda_1 \psi_{i-1} ( s+\theta_1) \frac{(k-i)\mu_2}{s+i\mu_1 + (k-i)\mu_2}\nonumber \\
  & & + \lambda_2 \psi_i ( s+\theta_2) \frac{(k-i)\mu_2}{s+i\mu_1 + (k-i)\mu_2} \nonumber\\
  & & + \lambda_2 \psi_{i+1} ( s+\theta_2) \frac{(i+1)\mu_1}{s+(i+1)\mu_1 + (k-1-i)\mu_2} \nonumber \\
  & & + \lambda_1 p_{i-1,k-1-i} + \lambda_2 p_{i,k-2-i}.
\end{eqnarray}
It is useful to rewrite the above equations in vector-matrix form. Let
\begin{eqnarray*}
\bpsi(s) &=& [\psi_0(s)\;\;\psi_1(s)\;\;\cdots\;\;\psi_{k-1}(s)],\\
\bphi(s) &=& \bpsi(s) - \bp_{k-1} \\
        &=& [\phi_0(s)\;\;\phi_1(s)\;\;\cdots\;\;\phi_{k-1}(s)].
\end{eqnarray*}
Then (\ref{eq:psii}) can be written as
\[
 s\bphi(s) = \bp_{k-1} \Delta_{k-1} - \bp_{k-2} \Lambda_{k-2} + \bpsi(s+\theta_1)A_1(s) + \bpsi(s+\theta_2)A_2(s)
\]
and by substituting \eqref{eq:bb1},
\begin{equation}\label{eq:phi}
 s\bphi(s) =
 \bp_{k-1} \Delta_{k-1} - \bp_{k-1} R_{k-1} \Lambda_{k-2} +
 \bpsi(s+\theta_1)A_1(s) + \bpsi(s+\theta_2)A_2(s)
\end{equation}
where the non-zero elements of the $k \times k$ matrices $A_1(s) = [a_{1,i,j}(s)]$ and $A_2(s) = [a_{2,i,j}(s)]$ are given by
\begin{eqnarray*}
a_{1,i,i}(s)   &=& \lambda_1\frac{s+(k-1-i)\mu_2}{s+(i+1)\mu_1+(k-1-i)\mu_2}, \quad 0 \le i \le k-1, \\
a_{1,i-1,i}(s) &=& -\lambda_1\frac{(k-i)\mu_2}{s+i\mu_1+(k-i)\mu_2},          \quad 0 <   i \le k-1, \\
a_{2,i,i}(s)   &=& \lambda_2\frac{s+i\mu_1}{s+i\mu_1+(k-i)\mu_2},             \quad 0 \le i \le k-1, \\
a_{2,i+1,i}(s) &=& -\lambda_2\frac{(i+1)\mu_1}{s+(i+1)\mu_1+(k-1-i)\mu_2},    \quad 0 \le i <   k-1.
\end{eqnarray*}
For $s > 0$, we can divide Eq. (\ref{eq:phi}) by $s$ and use the notation
\begin{equation} \label{eq:qs}
D(s) = I + \frac{\Delta_{k-1} - R_{k-1} \Lambda_{k-2}}s,   \quad
H_i (s) = \frac{A_i(s)}{s}, \quad i = 1, 2,
\end{equation}
to obtain
\begin{equation}\label{eq:psi0}
\bpsi (s) =  \bp_{k-1} D(s) +  \bpsi(s+\theta_1) H_1(s) + \bpsi(s+\theta_2) H_2 (s) ,
\end{equation}
This equation is suitable to recursively determine $\bpsi(s)$ for $s >0$.
Let
\[
D_{i,j}(s) =D(s+i\theta_1+j\theta_2), \quad \bpsi_{i,j}(s) = \bpsi(s+i\theta_1 + j\theta_2), \quad i,j \ge 0.
\]
Then Eq. \eqref{eq:psi0} yields for $i, j \ge 0$,
\begin{equation}
 \bpsi_{i,j}(s) = \bp_{k-1}D_{i,j}(s)
 + \bpsi_{i+1,j}(s)H_1(s+i\theta_1+j\theta_2)
 + \bpsi_{i,j+1}(s)H_2(s+i\theta_1+j\theta_2). \label{eq:psirec}
 \end{equation}
To obtain $\bpsi(s) = \bpsi_{0,0}(s)$ we can repeatedly apply the above equation:
\begin{eqnarray*}
 \bpsi_{0,0}(s) &=& \bp_{k-1} D_{0,0}(s) + \bpsi_{1,0}(s)H_1(s) + \bpsi_{0,1}(s)H_2(s) \\
 & = &
  \bp_{k-1} ( D_{0,0}(s)+ D_{1,0} (s) H_1(s) + D_{0,1} (s) H_2(s))
 + \bpsi_{2,0}(s)H_1(s+\theta_1)H_1(s) \\
 &&
 + \bpsi_{0,2}(s)H_2(s+\theta_2)H_2(s)
 + \bpsi_{1,1}(s) (H_2(s+\theta_1)H_1(s)+H_1(s+\theta_2)H_2(s))
\end{eqnarray*}
and so on.
This results after $n$ iterations in the following expression for $\bpsi(s)$:
\begin{equation} \label{eq:psis02}
\bpsi(s) = \bp_{k-1} \sum_{i+j < n} D_{i,j} (s) C_{i,j} (s) + \sum_{i+j=n} \bpsi_{i,j} (s) C_{i,j} (s) .
\end{equation}
The $k \times k$ matrices $C_{i,j}(s)$ are defined as follows. A sequence of grid points $\bp = \{(i_0,j_0), (i_1,j_1), \ldots, (i_n, j_n)\}$ is called a path from $(i_0,j_0)$ to $(i_n,j_n)$ if each of its steps $(i_{l+1}, j_{l+1}) - (i_{l},j_{l})$ are either $(1,0)$ or $(0,1)$. For path $\bp$, we introduce (see Fig. \ref{fig:path})
\[
C_\bp (s) = H_{i_{n-1},j_{n-1}} (s) \cdots H_{i_{1},j_{1}} (s) H_{i_{0},j_{0}} (s)
\]
where for $l = 0, \ldots, n-1$,
\[
H_{i_l,j_l} (s) = \left\{
\begin{array}{l l}
H_1 (s+i_l \theta_1 + j_l \theta_2) & \mbox{ if $(i_{l+1}, j_{l+1}) - (i_{l},j_{l}) = (1,0)$,}\\
H_2 (s+i_l \theta_1 + j_l \theta_2) & \mbox{ if $(i_{l+1}, j_{l+1}) - (i_{l},j_{l}) = (0,1)$.}\\
\end{array}
\right.
\]

\begin{figure}[htb]
\begin{center}
\includegraphics[width=0.6\linewidth]{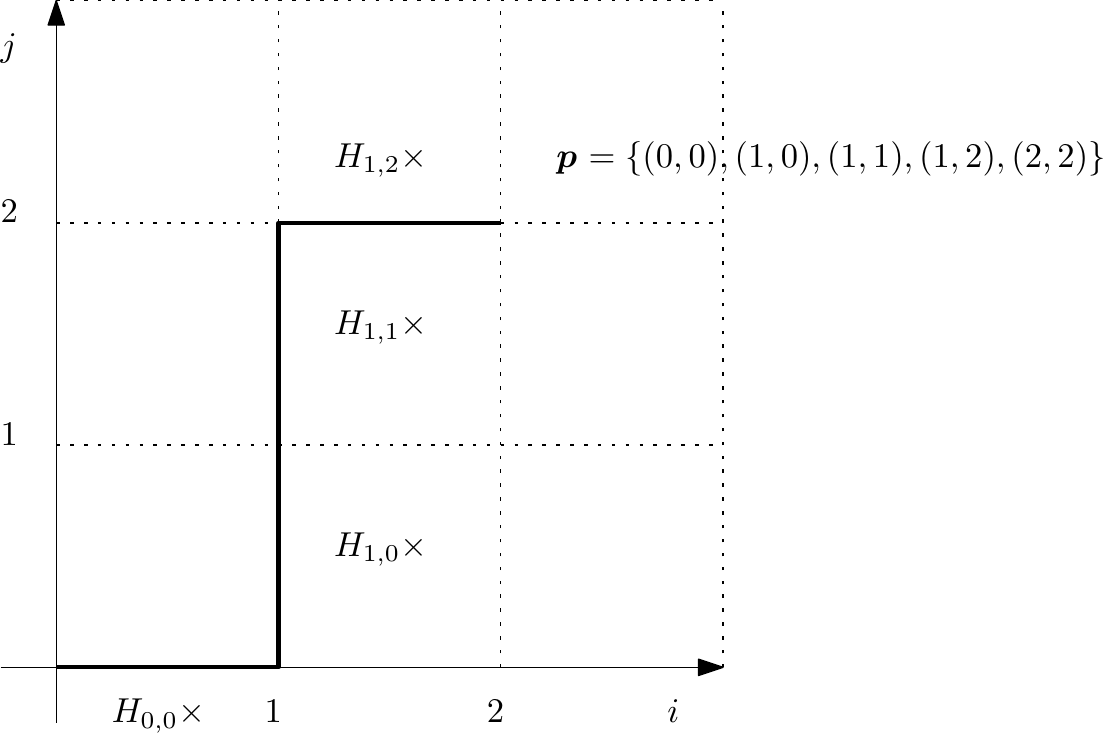}
\caption{Path $\bp = \{ (0,0), (1,0), (1,1), (1,2), (2,2) \}$ with $C_{\bp} (s) = H_{1,2}(s)H_{1,1}(s)H_{1,0}(s)H_{0,0}(s).$}
\label{fig:path}
\end{center}
\end{figure}

For path $\bp = \{ (i_0,j_0)\}$, we set $C_\bp (s) = I$.
Let $\mathcal{P}(i,j)$ be the set of all paths from $(0,0)$ to $(i,j)$. Then $C_{i,j} (s)$ is defined as
\[
C_{i,j} (s) = \sum_{\bp \in \mathcal{P}(i,j)} C_\bp (s) .
\]
Note that for $i+j > 0$, the $k \times k$ matrices $C_{i,j}(s)$ can be recursively calculated from
\begin{eqnarray}\label{eq:cij}
C_{i,j}(s) &=&  \sum_{\bp \in \mathcal{P}(i-1,j)} H_{i-1,j} (s) C_\bp (s) + \sum_{\bp \in \mathcal{P}(i,j-1)} H_{i,j-1} (s) C_\bp (s) \nonumber \\
& = &
{H}_1(s+(i-1)\theta_1+j\theta_2) C_{i-1,j}(s)  + {H}_2(s+ i\theta_1+(j-1)\theta_2) C_{i,j-1}(s),
\end{eqnarray}
where $C_{0,0} (s) = I$ and
$C_{i,j} (s)$ is the all zero matrix if $i<0$ or $j<0$.
The following lemma states that the series of $C_{i,j} (s)$ is well defined for all $s >0$.

\begin{lemma}\label{lem:C}
For each $\delta >0$, the series $\sum_{i=0}^\infty\sum_{j=0}^\infty C_{i,j}(s)$ is absolutely and uniformly convergent for all $s>\delta$.
\end{lemma}

\begin{proof}
Fix $\delta > 0$. It suffices to show that there are constants $M$ and $r < 1$ such that for all $s>\delta$ and $i+j \ge 0$,
\begin{equation}\label{eq:bound1}
  |C_{i,j}(s)| \le M r^{i+j} E,
\end{equation}
where $E$ is the all one matrix and the inequality is component wise. In this bound, $s$ needs to stay away from zero, since
$H_1(s)$ and $H_2 (s)$, and thus by \eqref{eq:cij}, also $C_{i,j} (s)$ are unbounded as $s$ approaches zero.
Now fix $r < 1$. Since $|H_l (s)| \le \lambda_l E / s$ ($l = 1, 2$), 
there is an $N \ge 0$ such that for $s > \delta$ and $i+j \ge N$,
\begin{equation}\label{eq:bound2}
|H_l (s+i\theta_1+j\theta_2)| \le \frac{r}{2k} E , \quad l = 1, 2.
\end{equation}
Recursion \eqref{eq:cij} implies that for each $i,j \ge 0$, $C_{i,j} (s)$ is bounded for $s > \delta$. Hence, there is a (sufficiently large) $M$ such that \eqref{eq:bound1} is valid for $s > \delta$ and the finitely many $i+j \le N$. By induction we now prove that \eqref{eq:bound1} is valid for all $i+j \ge N$. Suppose it holds for all $i+j = n$ (which is true for $n=N$).
From \eqref{eq:cij} and \eqref{eq:bound2} we get for $s > \delta$ and $i+j = n+1$,
\begin{eqnarray*}
  |C_{i,j}(s)| & \le &  |{H}_1(s+(i-1)\theta_1+j\theta_2)| |C_{i-1,j}(s)| +  |{H}_2(s+ i\theta_1+(j-1)\theta_2)| |C_{i,j-1}(s)|\\
  & \le & \frac{r}{2k} E \left(|C_{i-1,j}(s)| +  |C_{i,j-1}(s)| \right) \\
  & \le & \frac{r}{2k} 2 M r^{n} E^2 \; = \; M r^{n+1} E,
\end{eqnarray*}
where the last inequality follows from the induction hypothesis.\qed
\end{proof}

Since $D_{ij} (s)$ are uniformly bounded for all $s>\delta >0$ and $i+j \ge 0$, we immediately get the following.

\begin{corollary}\label{cor:abs}
The series $\sum_{i=0}^\infty\sum_{j=0}^\infty D_{i,j} (s) C_{i,j}(s)$ is absolutely and uniformly convergent for all $s>\delta >0$.
\end{corollary}

Using that $|\psi_{i,j}(s)| \le 1$, the second term in \eqref{eq:psis02} is bounded by
\[
\left| \sum_{i+j=n} \bpsi_{i,j} (s) C_{i,j} (s) \right| \le \sum_{i+j=n} | C_{i,j} (s) | .
\]
So it vanishes as $n \rightarrow \infty$ by virtue of the absolute convergence of the series of $C_{i,j} (s)$. Hence, taking $n \rightarrow \infty$ in \eqref{eq:psis02}, we get from
Corollary \ref{cor:abs},
\begin{equation} \label{eq:psis2}
\bpsi(s) = \bp_{k-1} (s)  C(s), 
\end{equation}
where
\begin{equation}\label{eq:Cs}
C(s) = \sum_{i=0}^\infty\sum_{j=0}^\infty D_{i,j} (s) C_{i,j}(s) .
\end{equation}
In particular, we have
\begin{equation}\label{eq:psitheta12}
\bpsi(\theta_i) = \bp_{k-1} C(\theta_i), \quad i = 1, 2.
\end{equation}
To complete the LST of the virtual queueing time, we need to determine $\bp_{k-1}$. First we set $s=0$ in (\ref{eq:phi}) yielding
\begin{eqnarray} \label{eq:pk}
0 & = & \bp_{k-1} \Delta_{k-1} - \bp_{k-1} R_{k-1}\Lambda_{k-2} + \sum_{i=1}^2 \bpsi(\theta_i)A_i(0)
         \nonumber \\
  & = & \bp_{k-1} \Delta_{k-1} - \bp_{k-1} R_{k-1}\Lambda_{k-2} +  \bp_{k-1} \sum_{i=1}^2 C(\theta_i)A_i(0),
\end{eqnarray}
where the second equality follows from \eqref{eq:psitheta12}.
To uniquely determine $\bp_{k-1}$, we finally need the normalizing equation
\begin{equation} \label{eq:nor}
\sum_{n = 0}^{k-1} \bp_{n}\be + \bphi(0) \be = 1,
\end{equation}
where $\be$ is the all one vector and $\bp_n$ is given by \eqref{eq:bb1} for $0\le n < k-1$.
However, Eq. (\ref{eq:nor}) requires the computation of $\bphi(0)$, which is the hard step. Taking the derivatives on both sides of (\ref{eq:phi}) and setting $s=0$, we get
\begin{equation}\label{eq:phi0}
\bphi(0) = \sum_{i=1}^2 \left( \bpsi(\theta_i)A_i'(0) + \bpsi'(\theta_i)A_i(0) \right).
\end{equation}
Here prime indicates derivative with respect to $s$. Thus, to calculate $\bphi(0)$ we need $\bpsi'(s)$ at $s=\theta_1$ and $s=\theta_2$. For this we can use (\ref{eq:psis2}). Differentiating (\ref{eq:psis2}) yields
\begin{equation}\label{eq:psis'}
\bpsi'(s) =  \bp_{k-1} C'(s) = \bp_{k-1}  \sum_{i=0}^\infty\sum_{j=0}^\infty \left( D'_{i,j} (s) C_{i,j} (s) + D_{i,j} (s) C'_{i,j}(s) \right).
\end{equation}
The terms $C'_{i,j} (s)$ can be recursively computed by differentiating (\ref{eq:cij}),
\begin{eqnarray} \label{eq:cpij}
C'_{i,j}(s) & = &  {H}_1(s+(i-1)\theta_1+j\theta_2)C'_{i-1,j}(s) + {H}_2(s+ i\theta_1+(j-1)\theta_2) C'_{i,j-1}(s) \nonumber \\
            &   &  + {H}'_1(s+(i-1)\theta_1+j\theta_2)C_{i-1,j}(s)  +  {H}'_2(s+ i\theta_1+(j-1)\theta_2) C_{i,j-1}(s) , \nonumber
\end{eqnarray}
where
$C'_{i,j} (s)$ is the all zero matrix if $i=j=0$ or if $i<0$ or $j<0$.
Term by term differentiation of (\ref{eq:Cs}) is justified by the following two lemmas.

\begin{lemma}\label{lem:Cp}
For each $\delta >0$, the series $\sum_{i=0}^\infty\sum_{j=0}^\infty C'_{i,j}(s)$ is absolutely and uniformly convergent for all $s>\delta$.
\end{lemma}

\begin{proof}
The proof is similar to the proof of Lemma \ref{lem:C}.
Fix $\delta > 0$. It suffices to show that there are constants $M$ and $r < 1$ such that for all $s>\delta$ and $i+j \ge 0$,
\begin{equation}\label{eq:bound3}
  |C'_{i,j}(s)| \le M r^{i+j} E.
\end{equation}
First fix $r < 1$. Since $H_l (s) \le \lambda_l E / s$ and $H'_l (s) \le \lambda_l E / s^2$ ($l = 1, 2$), there is an $N \ge 0$ such that for $s > \delta$ and $i+j \ge N$,
\begin{equation}\label{eq:bound4}
|H_l (s+i\theta_1+j\theta_2)| \le \frac{r}{4k} E , \quad
|H'_l (s+i\theta_1+j\theta_2)| \le \frac{r}{4k} E , \quad l = 1,2 .
\end{equation}
Recursions \eqref{eq:cij} and \eqref{eq:cpij} imply that for each $i,j \ge 0$, $C_{i,j} (s)$ and $C'_{i,j} (s)$ are bounded for $s > \delta$. Hence, there is a (sufficiently large) $M$ such that both \eqref{eq:bound1} and \eqref{eq:bound3} are valid for $s > \delta$ and the finitely many $i+j \le N$. Following the induction steps in the proof of Lemma \ref{lem:C} it follows that \eqref{eq:bound1} is valid for all $i+j \ge N$. We now show that also \eqref{eq:bound3} holds for all $i+j \ge N$.
Suppose that \eqref{eq:bound3} holds for all $i+j = n$ (which is true for $n=N$).
From \eqref{eq:cpij} and \eqref{eq:bound4} we get for $i+j = n+1$,
\begin{eqnarray*}
  |C'_{i,j}(s)| & \le & |{H}_1(s+(i-1)\theta_1+j\theta_2)||C'_{i-1,j}(s)|  +  |{H}_2(s+ i\theta_1+(j-1)\theta_2)||C'_{i,j-1}(s)| \\
                &  +  & |{H}'_1(s+(i-1)\theta_1+j\theta_2)||C_{i-1,j}(s)|  + |{H}'_2(s+ i\theta_1+(j-1)\theta_2)| |C_{i,j-1}(s)| \\
  & \le & \frac{r}{4k} E \left(|C'_{i-1,j}(s)| +  |C'_{i,j-1}(s)| + |C_{i-1,j}(s)| +  |C_{i,j-1}(s)| \right) \\
  & \le & \frac{r}{4k} 4 M r^{n} E^2 \; = \; M r^{n+1} E,
\end{eqnarray*}
where the last inequality follows form the induction hypothesis. \qed
\end{proof}

\begin{lemma}\label{lem:Cp2}
For each $s >0$, the derivative of $C(s)$ exists and is equal to
\[
C'(s) = \sum_{i=0}^\infty\sum_{j=0}^\infty \left( D'_{i,j} (s) C_{i,j} (s) + D_{i,j} (s) C'_{i,j}(s) \right) .
\]
\end{lemma}

\begin{proof}
Fix $s > 0$. First note that the series converges by Lemmas \ref{lem:C}-\ref{lem:Cp} and the fact that $D_{i,j} (s)$ and $D'_{i,j} (s)$ are uniformly bounded for all $i+j \ge 0$. It suffices to prove that for each sequence $\{h_n\}$ converging to $0$, such that $s+h_n > 0$ for all $n$,
\[
\lim_{n\rightarrow \infty} \frac{C(s+h_n)-C(s)}{h_n} = \sum_{i=0}^\infty\sum_{j=0}^\infty B'_{i,j} (s)
\]
where $B_{i,j} (s) = D_{i,j} (s) C_{i,j} (s)$.
Let \{$h_n$\} be such sequence. Thus there is a $\delta > 0$ such that $s+h_n > \delta$ for all $n$. According to \eqref{eq:bound1} and \eqref{eq:bound3} and that the fact that
$D_{i,j} (s+h_n)$ and $D'_{i,j} (s+h_n)$ are uniformly bounded for all $n$ and $i+j \ge 0$,
there are constants $M$ and $r < 1$ such that for all $n$ and $i+j \ge 0$,
\begin{equation}\label{eq:bound5}
  |B'_{i,j}(s+h_n)| \le M r^{i+j} E.
\end{equation}
We need to show that for each $\epsilon >0$ there is an $N$ such that for $n > N$
\begin{equation}\label{eq:bound6}
\left| \frac{C(s+h_n)-C(s)}{h_n} - \sum_{i=0}^\infty\sum_{j=0}^\infty B'_{i,j}(s) \right|
=
\left|  \sum_{i=0}^\infty\sum_{j=0}^\infty \left( \frac{B_{i,j}(s+h_n)-B_{i,j}(s)}{h_n} - B'_{i,j}(s) \right) \right| < \epsilon .
\end{equation}
Let $\epsilon > 0$. For given $n$ and $i,j \ge 0$, it follows from the mean value theorem that there is an $0<\eta<1$ such that
$
(B_{i,j}(s+h_n)-B_{i,j}(s))/h_n = B'_{i,j} (s + \eta h_n)
$.
Hence, by \eqref{eq:bound5},
\[
\left| \frac{B_{i,j}(s+h_n)-B_{i,j}(s)}{h_n} - B'_{i,j}(s) \right| \le |B'_{i,j} (s + \eta h_n)| + |B'_{i,j}(s)|
\le 2 M r^{i+j} E.
\]
Note that $M$ and $r$ do not depend on $n, i, j$. So there is a constant $K$ such that for all $n$,
\begin{equation}\label{eq:bound7}
\sum_{i+j\ge K} \left| \frac{B_{i,j}(s+h_n)-B_{i,j}(s)}{h_n} - B'_{i,j}(s) \right| < \frac{\epsilon}2 .
\end{equation}
Further, for given $i, j\ge 0$,
$(B_{i,j}(s+h_n)-B_{i,j}(s))/h_n$ converges to $B'_{i,j} (s)$ as $n$ tends to infinity. Hence, there is an $N$ such that for $n > N$
\begin{equation}\label{eq:bound8}
\sum_{0 \le i+j <  K} \left| \frac{B_{i,j}(s+h_n)-B_{i,j}(s)}{h_n} - B'_{i,j}(s) \right| < \frac{\epsilon}2 .
\end{equation}
Combining \eqref{eq:bound7} and \eqref{eq:bound8} yields \eqref{eq:bound6}. \qed
\end{proof}

Substitution of \eqref{eq:psitheta12} and \eqref{eq:psis'} with $s = \theta_i$ in \eqref{eq:phi0} yields
\[
\bphi(0) = \bp_{k-1} \sum_{i=1}^{2}
\left( C(\theta_i) A'_i (0) + C'(\theta_i) A_i (0) \right)
\]
and normalization Eq. (\ref{eq:nor}) can be rewritten as
\begin{equation}\label{eq:nor2}
\sum_{n = 0}^{k-1} \bp_{n}\be + \bp_{k-1} \sum_{i=1}^{2}
\left( C(\theta_i) A'_i (0) + C'(\theta_i) A_i (0) \right) \be = 1.
\end{equation}

The above findings are summarized in the following theorem.

\begin{theorem}
  The steady-state LST $\bpsi (s)$ of the virtual queueing time satisfies
  \[
  \bpsi(s) = \bp_{k-1} C(s),
  \]
  where $C(s)$ is defined by \eqref{eq:Cs} and the probability vectors $\bp_n$ for $0 \le n \le k-1$ are the solution to the system of linear equations \eqref{eq:bb1}, \eqref{eq:pk} and \eqref{eq:nor2}.
\end{theorem}

\subsection{Special Case $\mu_1 = \mu_2 = \mu$}\label{sec:Special}
We now assume $\mu_1=\mu_2=\mu$. This case has also been studied by Sarhangian and Balc{\i}og̃lu \cite{sarhangian2013waiting}. The problem simplifies considerably, since we do not need to keep track of $N_1(t)$ and $N_2(t)$ separately, only $N(t) = N_1(t)+N_2(t)$. Define
\begin{eqnarray*}
\psi(s) & = & \lim_{t \rightarrow \infty} E(s^{-sW(t)};N(t) = k-1)\\
& = & E(s^{-sW};N = k-1), \\
p_{i} & = & \lim_{t \rightarrow \infty} P(W(t) = 0, N(t)=i) \\
& = & P(W=0, N = i), \quad 0 \le i \le k-1,\\
\phi(s) & = & \psi(s) - p_{k-1} .
\end{eqnarray*}
Then the balance equations \eqref{eq:b1}  can be simplified to
\begin{equation}\label{eq:b2}
p_{n-1}  (\lambda_1+\lambda_2) = p_n n \mu, \quad 1 \le n \le k-1,
\end{equation}
and  \eqref{eq:psi0} reduces to
\[
\psi (s) = p_{k-1} + \psi(s+\theta_1) \frac{\lambda_1 }{s+k\mu} + \psi(s+\theta_2) \frac{\lambda_2 }{s+k\mu} .
\]
The solution of this equation is given by (cf. \eqref{eq:psis2})
\[
\psi(s) = p_{k-1} c(s),
\]
where
\[
c(s) = \sum_{i=0}^\infty\sum_{j=0}^\infty c_{i,j}(s) .
\]
For $i + j > 0$ the terms $c_{i,j}(s)$ are determined from the recursion
\[ {c_{i,j}(s)} = \frac{\lambda_1}{s+(i-1)\theta_1+j\theta_2 + k\mu}{c_{i-1,j}(s)} + \frac{\lambda_2}{s+ i\theta_1+(j-1)\theta_2+k\mu}{c_{i,j-1}(s)}\]
with $c_{0,0} (s) = 1$ and $c_{i,j} (s) = 0$ if $i<0$ or $j<0$. 
The normalization equation becomes
\[
1 = \sum_{n=0}^{k-1} p_n + \phi(0) = \sum_{n=0}^{k-1} p_n + \psi(\theta_1) \frac{\lambda_1 }{k\mu} + \psi(\theta_2) \frac{\lambda_2 }{k\mu}.
\]
Together with \eqref{eq:b2} this yields for $n = 0, \ldots, k-1$,
\[
p_n = \left(1-\psi(\theta_1) \frac{\lambda_1 }{k\mu} - \psi(\theta_2) \frac{\lambda_2 }{k\mu}\right) \frac{\frac{\rho^n}{n!}}{\sum_{j=0}^{k-1} \frac{\rho^j}{j!}} ,
\]
where $\rho = \frac{\lambda_1+\lambda_2}{\mu}$.

\section{Performance Measures} \label{sec:pm}
Now we show how many useful performance measures in steady state can be computed in terms of the LST evaluated at $\theta_1$ and $\theta_2$. Suppose the $M/M/k+M$ system is in steady state. An arrival faces a queuing time of $W$. If the arrival is from class $i$, he will enter service if his impatience time $T_i$ is longer than $W$. For class 1, this probability is given by (and similarly for class 2, by replacing $\theta_1$ by $\theta_2$)
\[
 P(T_1 > W) = E(e^{-\theta_1 W})  =  \sum_{i+j \le k-1}  E(e^{-\theta_1W};N_1 = i, N_2 = j)
  = \sum_{n < k-1} \bp_n \be + \bpsi(\theta_1) \be.
 \]
Here we are using that $W=0$ and $E(e^{-\theta W};N_1 = i, N_2 = j) = p_{ij}$ when $i+j < k-1$. Next, using Little's law, we see that the expected number of servers busy serving class $1$ customers is given by
\[\lambda_1 P(T_1 > W) \frac{1}{\mu_1} = \frac{\lambda_1}{\mu_1} \left( \sum_{n < k-1} \bp_n \be + \bpsi(\theta_1) \be \right)\]
and the steady state throughput is equal to
\begin{eqnarray*}
 \sum_{i=1}^2 \lambda_i P(T_i > W) &=& \sum_{i=1}^2 \lambda_i \left( \sum_{n < k-1} \bp_n \be + \psi(\theta_i) \be \right).
\end{eqnarray*}
Next we compute the expected time of a class $1$ customer waiting for service. This is given by
\begin{eqnarray*}
E(\min(W,T_1)) &=& E(E(\min(W,T_1)|W))\;  = \; E\left(\frac{1-e^{-\theta_1 W}}{\theta_1}\right) \\
&=& \frac{1-E(e^{-\theta_1 W})}{\theta_1} \; = \; \frac{1-P(T_1 > W)}{\theta_1}.
\end{eqnarray*}
By Little's law, we obtain for the expected number of class $i$ customers waiting for service
\[
E(L_i^q) = \lambda_i E(\min(W,T_i)) = \frac{\lambda_i}{\theta_i} (1-P(T_i > W)).
\]
Finally, the expected conditional waiting time of class $i$ customers entering service follows from
\[
E(W| T_i > W) = \frac{E(W; T_i > W)}{P(T_i > W)}
\]
where
\[
E(W; T_i > W) = E(W e^{-\theta_i W}) = \left.\frac{d}{ds} E(e^{-s W})\right|_{s=\theta_i} = \bpsi'(\theta_i) \be.
\]
These formulas simplify considerably when applied to $M/G/1+M$ system.  In particular, the probability that the server is busy serving a class $i$ customer is given by
\[ \rho_i =  \lambda_i\tau_i\psi(\theta_i),\]
where $\psi(s)$ is as defined in Eq.~\eqref{eq:psidef}.
The probability that the server is busy is given by
\[ \rho = \psi(\theta_1)\lambda_1\tau_1 + \psi(\theta_2)\lambda_2\tau_2.\]
In steady state the throughput is given by
\[ \lambda_1\psi(\theta_1) + \lambda_2\psi(\theta_2)\]
and the reneging rate by
\[ \lambda_1(1-\psi(\theta_1)) + \lambda_2(1-\psi(\theta_2)).\]
The expected number of class $i$ customers waiting for service in steady state is given by
\[ E(L_i^q) = \frac{\lambda_i}{\theta_i}(1-\psi(\theta_i)).\]
The expected number of class $i$ customers in the system in steady state is given by
\[ E(L_i) = E(L_i^q) + \lambda_i\tau_i\psi(\theta_i).\]
This implies that, in the special case when $\tau_i = 1/\theta_i$,
\[ E(L_i) = \frac{\lambda_i}{\theta_i}.\]
This is expected, since in this case the system behaves like an infinite server queue for each class of customers.

\section{Numerical Analysis}\label{sec:Numerical}
Recall that the previous analytical models of multiclass FCFS systems have allowed for customers across classes to differ in either their service time distributions or their patience time distributions, but not in both distributions (\cite{van2012analysis}, \cite{sakuma2017multi}, \cite{sarhangian2013waiting}). A contribution of our work is that we allow customers across classes to differ in both distributions. This enables us to model service systems such as call centers that segment their arrivals into classes of callers whose requests may differ greatly in their complexity and criticality. We are therefore interested in using our characterization to compare the performance of a system where customers across classes differ in only one of the distributions with the performance of a system where the customers across classes differ in both distributions. To do this, we conduct a numerical analysis by using the performance measures that we derived for the $M$/$M$/$k$+$M$ system in Sect. \ref{sec:pm}.

We compare the performance of three systems over a range of system loads. In all of the systems, requests from class 1 have a mean service time of 1 unit ($\mu_1=1$) and requests from class 2 have a mean service time of .5 units ($\mu_2=2$). In the first system, customers in both classes are equally patient, with a mean patience time of 2/3 units ($\theta_1,\theta_2=1.5$). We call this system the \textit{base} system as customers from the two classes differ in their distribution of service times but not in their distribution of patience times. In the second system, the mean patience time of class 1 customers is 1 unit ($\theta_1=1$), while the mean patience time of class 2 customers is .5 units ($\theta_2=2$). We call this system the \textit{positive} system as there is a positive correlation across classes between customers' service times and patience times, i.e., class 1 customers have longer average service times and longer mean patience times. In the third system, class 1 customers have a mean patience time of .5 units ($\theta_1=2$), while class 2 customers have a mean patience time of 1 unit ($\theta_2=1$). Holding with our naming convention, we call this system the \textit{negative} system. In each scenario in our analysis, the arrival rate of the class 1 and the class 2 customers are equal ($\lambda_1=\lambda_2$). To vary the system load, we hold the number of servers in the system at 5, while varying the total arrival rate ($\lambda=\lambda_1+\lambda_2$) between 6 and 20.

In Figure \ref{fig:analysis} we present four steady-state performance measures from the three systems. The measures include the percentage of all customers who receive service, the mean waiting time of all customers, the system throughput, and the average service time of customers who receive service. We discuss each of these measures:

\begin{figure}[h]
\begin{center}
	\includegraphics[trim = 0.02in 0in 4.6in 0in, clip, width=6.5 in]{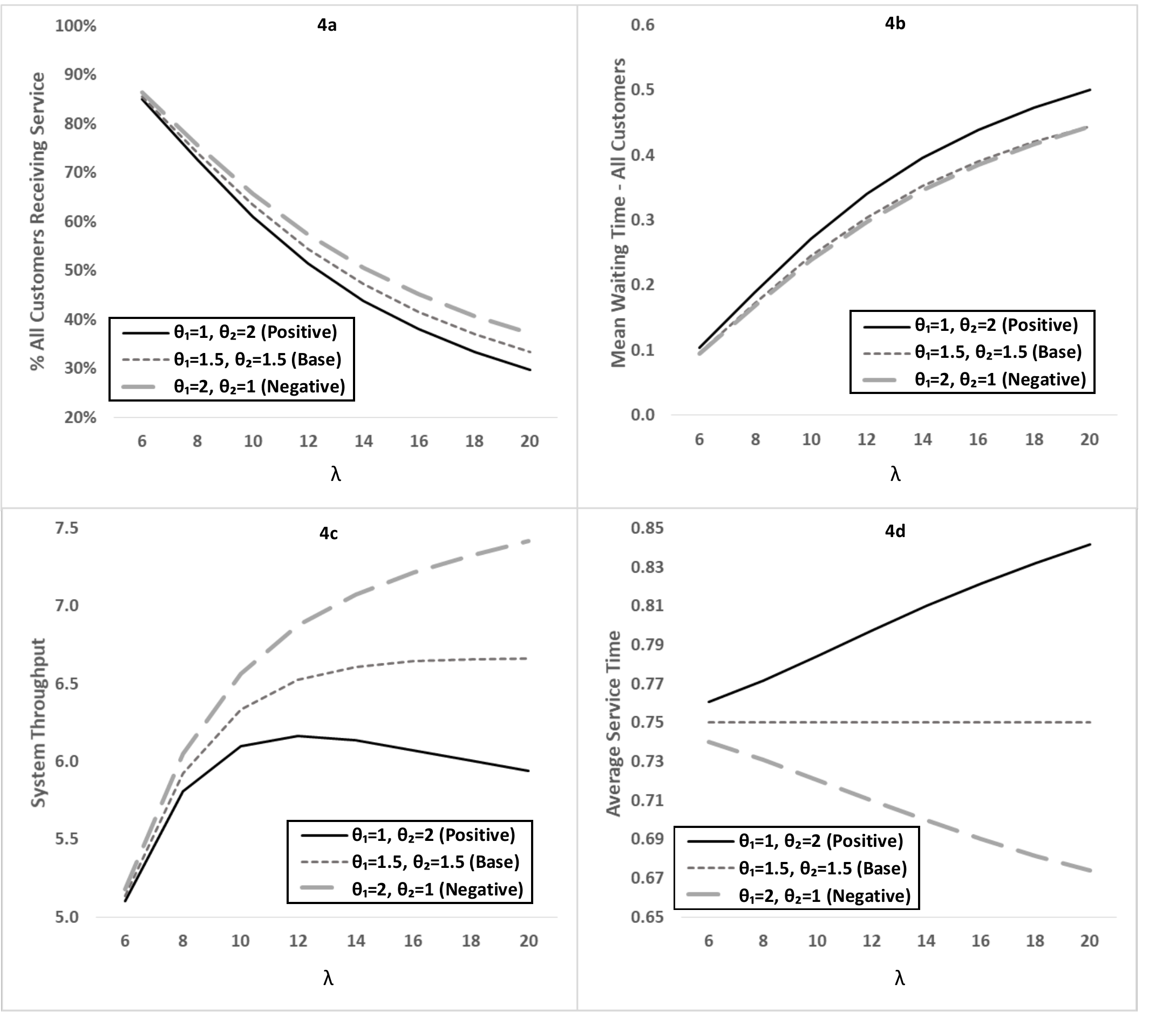}
	\caption{Various Performance Measures of Numerical Analysis}
	\label{fig:analysis}
\end{center}
\end{figure}

\begin{itemize}

\item \textbf{\% All Customers Receiving Service (4a):} The percentage of all customers who receive service is lowest in the positive system and highest in the negative system over all arrival rates. Recall that class 1 customers are more patient than class 2 customers in the positive system and vice versa in the negative system. Consequently, a greater (smaller) proportion of the customers who receive service are from class 1 in the positive (negative) system. Also recall that class 1 customers have longer average service times than class 2 customers. Since servers in the positive (negative) system spend a greater (smaller) percentage of their time serving the customers with longer service times, their aggregate service rate decreases (increases), which decreases (increases) the percentage of all customers who receive service. This result has ramifications for service level forecasting as one of the common measures of service level is the percentage of customers who receive service. This demonstrates that managers who do not account for differences in the distribution of customers' service times and patience times across classes may produce inaccurate service level forecasts.

\item \textbf{Mean Waiting Time of All Customers (4b):} Irrespective of the arrival rate, customers have the highest mean waiting times in the positive system. However, the mean waiting times of customers in the base system and the negative system are nearly the same. This result was surprising as we expected the average waiting times of the customers in the negative system to be lowest since the average waiting times of customers is highest in the positive system. Because this trend was puzzling to us, we calculate the mean waiting times conditional on receiving service and on reneging for both systems. We find two trends. First, in both of the systems the average waiting times of customers who receive service is higher than the average waiting times of customers who renege. Second, both of these waiting time measures are lower in the negative system than in the base system. So, if both of these measures are lower in the negative system, why are the average waiting times of all customers nearly equal between the two systems? The answer lies in the fact that a higher percentage of customers receive service in the negative system. One can think of the average waiting time of all customers as a weighted average of the waiting times of the customers who receive service and of the customers who renege, where the respective weights are the percentage of customers who receive service and the percentage of customers who renege. Recall from Figure 4a that the percentage of customers who receive service is highest in the negative system. This means that in the negative system there is a greater weight from a subset of customers who tend to wait longer, since the average waiting time for customers who receive service is greater than the average waiting time of customers who renege. The result is that the average waiting times of all customers in the base system and the negative system are nearly equal even though the average waiting times of customers who receive service and the average waiting times of customers who renege are both lower in the negative system.\footnote{For example, in the base system with $\lambda=20$, the average waiting time of customers who receive service is 0.654 and the average waiting time of customers who renege is 0.337. In the negative system, the respective averages are 0.641 and 0.324. However, in the base system the percentage of customers receiving service is 33.4\%, while it is 37.2\% in the negative system. Hence, the total average waiting time in the base system is $0.334*0.654 + (1-0.334)*0.337 = 0.443$, and the total average waiting time in the negative system is $0.372*0.641 + (1-0.372)*0.324 = 0.442$, which are nearly equal.} This result again has ramifications for service level forecasting as the average waiting time of all customers is another common measure of service level.

\item \textbf{System Throughput (4c):} We observed in Figure 4a that the percentage of customers receiving service is lowest (highest) under the positive (negative) system. It is therefore not surprising that system throughput is lowest (highest) under the positive (negative) system. However, what is surprising is that in the positive system, throughput is first increasing in the arrival rate, but is then decreasing after some threshold arrival rate. Initially, increasing the arrival rate increases server utilization and hence system throughput. However, the gains in throughput due to the increase in server utilization diminish and are eventually offset by a reduction in the effective service rate of the system. The reason that the service rate is decreasing in system load is that the percentage of time the servers spend with customers from class 1 (who take longer to serve) is increasing in load. This is because a higher proportion of class 2 customers renege as load increases since they are less patient than class 1 customers. The interesting takeaway is that in a system where customers' service times and patience times are positively correlated, increasing traffic can actually decrease throughput. This result has potential implications for systems with limited service capacity that generate revenue based on system throughput, e.g., restaurants. Mangers of such systems may attempt to increase traffic through marketing efforts in order to generate additional revenue but may instead reduce their revenue if the customers in their system who take longer to serve are also more patient.

\item \textbf{Average Service Time (4d):} In the base system the average service time of customers who receive service remains the same regardless of the arrival rate. This is because customers in each class are equally patient, which makes the proportion of customers receiving service who are from each class invariant to the system load. However, in the positive (negative) system, the average service time increases (decreases) as the arrival rate increases. In the positive system, the class 1 customers are more patient. Thus, as the load on the system increases, the proportion of customers receiving service who are from class 1 increases. Because class 1 customers take longer to serve, the average service time increases as the arrival rate increases. The reverse is true in the negative system, where class 2 customers are more patient but take less time to serve. This relationship between system load and average service time has managerial implications. Servers in customer-facing systems are often evaluated and incentivized based on their average service times. For example, a common practice in call centers is to provide financial rewards for agents to keep their average service time under some threshold. Our analysis shows that managers may reach false conclusions regarding their servers' productivity if they evaluate their servers based solely on their average service times. In the negative system, managers may receive the impression that servers are speeding up when the arrival rate is higher. Consequently, managers may wrongly reward their servers for their supposed efforts to work faster under heavy loads. Conversely, in the positive system, managers may falsely reprimand their servers for allegedly slowing down as arrival rates increase. Our model demonstrates than an observed correlation between average service times and system load, such as what we see in Figure 4d, may have no correlation with the servers' efforts. Rather, the observed correlation may be entirely due to a correlation across classes between customers' service times and patience times.

\end{itemize}

We make one final observation regarding the performance of these systems as the arrival rate tends to infinity. We've shown that of the customers who receive service, the proportion who are from the more patient class increases as the arrival rate increases. Intuitively, we would expect this proportion to tend to one as the arrival rate tends to infinity due to the stochastic dominance of the patience times of one class over the other. Indeed, as we increase the arrival rate in the negative and positive systems from our numerical examples, we observe this phenomenon. For example, when we set the arrival rate to 1,000 for each class in the positive system, of the callers who receive service, 95.9\% are from class 2. An outcome of this phenomenon is that as the arrival rate tends to infinity the average service time and throughput of the system tends to the same performance measures of an overloaded system comprised entirely of customers from the more patient class. Hence, in the positive system the average service time tends to 1 and system throughput tends to 5, while in the negative system average service time tends to .5 and system throughput tends to 10.

Overall, our numerical analysis demonstrates the importance of accounting for differences across classes in the distribution of customers' service times and patience times. We see that differences in these distributions may substantially affect key performance measures, which have a wide range of managerial implications, including service level forecasting, revenue management, and the evaluation of server performance. Consequently, a contribution of our work is that our analytical characterizations may be used to demonstrate to managers some of the implications of administering a multiclass FCFS queue.

\section{Model as Performance Approximation}\label{sec:Approx}
Managers may be interested in knowing whether they may use the performance measures from our model to approximate the performance of their real world systems. Thus, as a final exercise we compare the simulated performance of a system based on real data with the performance of a comparable $M$/$M$/$k$+$M$ system. Our data comes from a multiclass call center of a small US-based bank. To construct our simulated system, we select two classes from the data which differ in their distribution of service times and caller patience times. The first class is comprised of general banking calls such as balance inquiries and the second class is comprised of technical support calls such as password reset requests. In Figure \ref{fig:service times} we display the estimated density function of the service times from each class and in Figure \ref{fig:patience times} we display the estimated distribution function of the callers' patience times from each class.\footnote{Because we do not know the patience times of the callers in our data who received service, the patience time data is right-censored. Thus, we estimate the patience time distribution using the Kaplan Meier estimator, which accounts for this form of censoring. However, since waiting times in this call center were not long enough to reveal the entire distribution of patience times, the Kaplan Meier estimator only estimates a portion of the distribution. To fill in the remainder of the distribution, we assume that callers' patience times are exponentially distributed with rate parameters equal to the callers' reneging rate over the estimated portion of the distribution.} Note that class 1 callers tend to have shorter service times and shorter patience times than class 2 callers\footnote{The mean service time of callers from class 1 and class 2 are 223.97 and 448.82 seconds, respectively. The mean patience times are 394.08 and 946.53 seconds, respectively.}, i.e., there is a positive correlation across the classes between service times and patience times. To make the systems comparable, we set the mean service time and mean patience time of the two classes in the $M$/$M$/$k$+$M$ system equal to the estimated means of the two classes in the simulated system. Finally, we are interested in understanding to what extent the accuracy of the approximation increases when service rates are allowed to differ across classes. Thus, we also obtain the performance measures of this system under the restriction that the service rates across classes are equal. To do this, we calculate the average service rate across both classes and use the formulas from Sect. \ref{sec:Special} to obtain the performance measures.

\begin{figure}[h]
\begin{center}
	\includegraphics[trim = 0.0in 0in 1.6in 0in, clip, width=4.5 in]{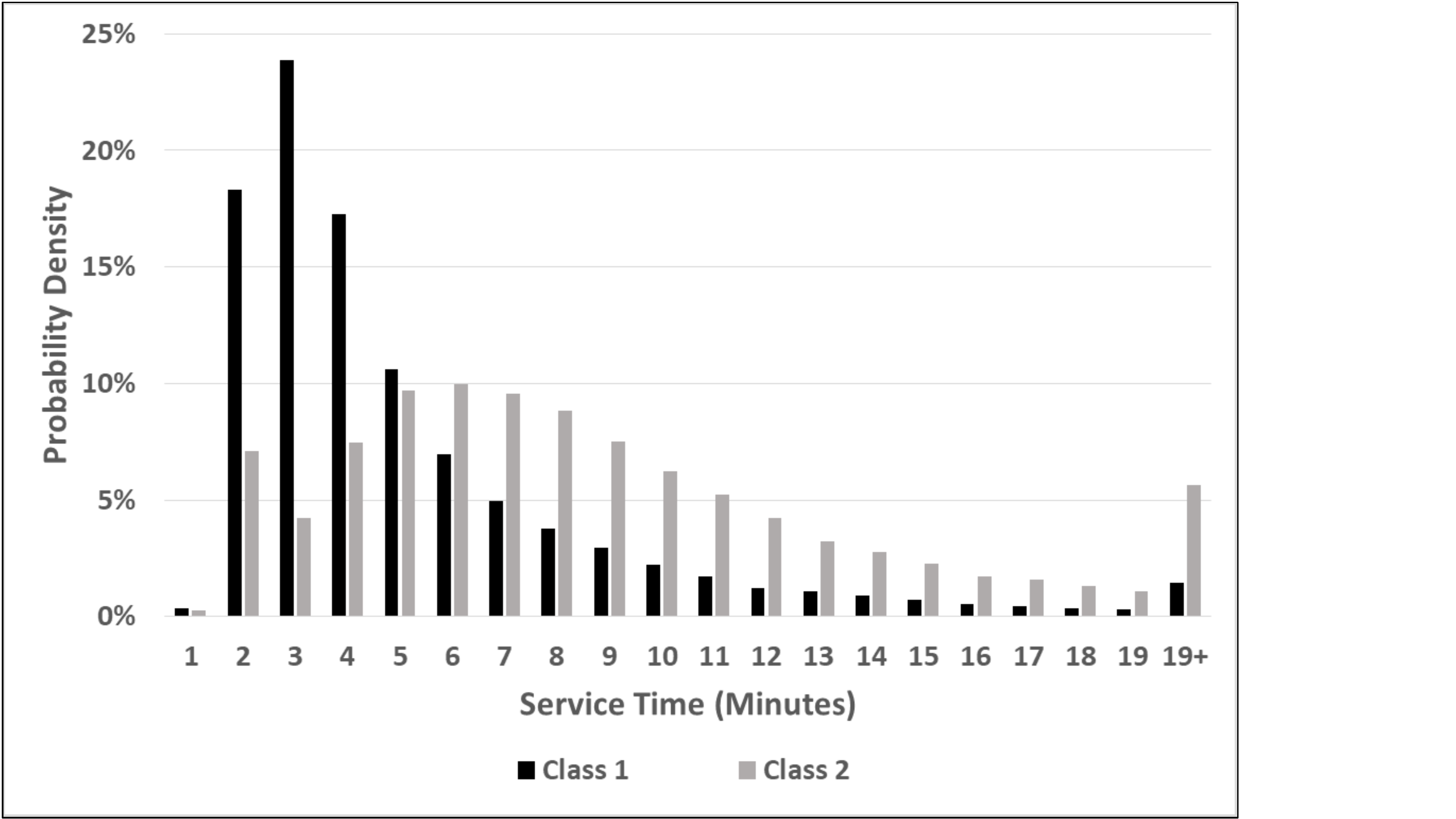}
	\caption{Empirical Density Function of Service Times from two Call Center Classes}
	\label{fig:service times}
\end{center}
\end{figure}

\begin{figure}[h]
\begin{center}
	\includegraphics[trim = 0.0in 0in 1.35in 0in, clip, width=4.5 in]{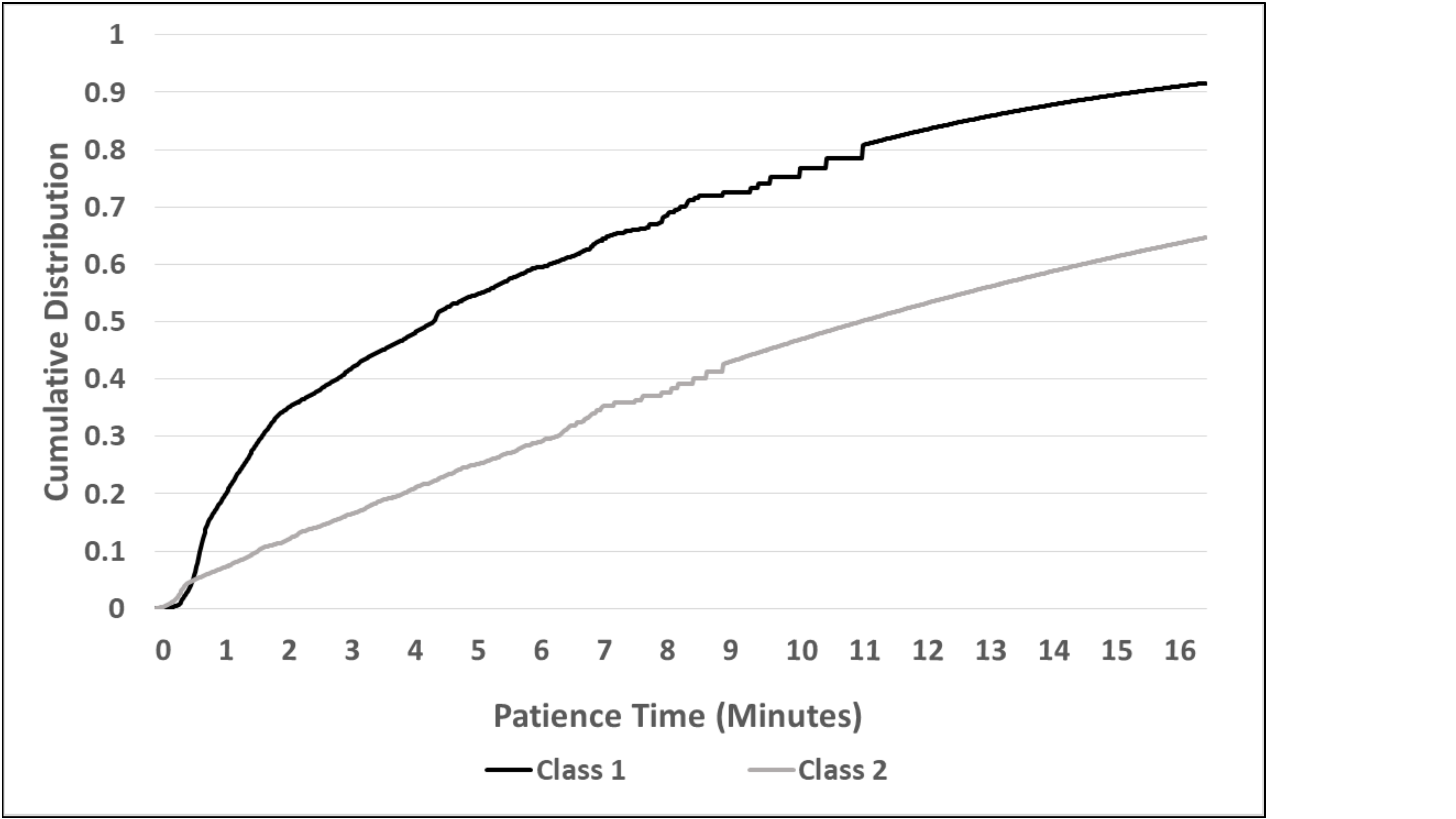}
	\caption{Empirical Distribution Function of Caller Patience Times from two Call Center Classes}
	\label{fig:patience times}
\end{center}
\end{figure}

In all of the systems we assume that calls from each class arrive according to independent Poisson processes with equal arrival rates. To compare the performance of the systems across different loads, we hold the number of servers in the system at 5, and vary the total arrival rate to be 36, 45, 60, and 120 calls per hour. In Table \ref{tab:comparison} we compare the performance of the three systems, where ``Simulation'' corresponds to the simulated system, ``$\mu_1\neq \mu_2$'' corresponds to the analytical characterization of the system where service rates are allowed to differ across classes and ``$\mu_1=\mu_2$'' corresponds to the analytical characterization where we restrict the service rates to be the same across classes. For each class, we present the average waiting time in seconds (AWT), the percentage of callers who receive service (\%RS), and the average number of callers waiting in queue (AQ). We also present the server utilization (Util \%), and the average service time of callers who receive service (AST(All)). We measure how close our analytical characterizations come to the simulated system using relative error, which is given by $|\textrm{simulated} - \textrm{analytical}|/\textrm{simulated}.$ Overall, we find that the performance measures of the $\mu_1\neq \mu_2$ system serve as good approximations of the simulated system, with relative errors of no greater than 3.76\%. Of particular note is the high accuracy in predicting the percentage of callers who receive service (\%RS) and the average service time of callers who receive service (AST(All)), with relative errors typically less than 1\%. Under the lowest arrival rate ($\lambda=36$), the $\mu_1=\mu_2$ characterization provides nearly the same approximation accuracy as the $\mu_1\neq \mu_2$ characterization. However, as the arrival rate increases, the accuracy of the $\mu_1 = \mu_2$ characterization decreases relative to the $\mu_1\neq \mu_2$ characterization. In particular the $\mu_1 = \mu_2$ characterization underforecasts AWT, AQ, Util\% and AST(All) while overforecasting \%RS. This is due to the fact that in this system, callers who tend to be more patient also tend to have longer service times. While the $\mu_1\neq \mu_2$ characterization accounts for this correlation, the $\mu_1 = \mu_2$ characterization does not. Overall, these results demonstrate that managers of two-class FCFS service systems may use our analytical characterization of the $\mu_1\neq \mu_2$ system to produce good approximations of the performance of their systems by collecting the mean service time and the mean patience time of customers in each class.

\begin{table}[htbp]
  \centering
  \caption{Performance Comparison of Simulated System with Real Data and $M$/$M$/$k$+$M$ (Analytical) system}
    \begin{tabular}{lc|cc|cc|cc|c|c}
    \toprule
          &       & \multicolumn{2}{c|}{\textbf{AWT (sec.)}} & \multicolumn{2}{c|}{\textbf{\%RS}} & \multicolumn{2}{c|}{\textbf{AQ}} &       &  \\
    \textbf{System} & $\boldsymbol\lambda$ \textbf{(hr.)} & 1     & 2     & 1     & 2     & 1     & 2     & \textbf{Util \%} & \textbf{AST(All)} \\
        \midrule
    Simulation & 36 & 26.98 & 31.43 & 93.23\% & 96.74\% & 0.13  & 0.16  & 64.44\% & 339.32 \\
    $\mu_1\neq \mu_2$ & 36 & 27.92 & 32.56 & 92.92\% & 96.56\% & 0.14  & 0.16  & 64.15\% & 338.56 \\
    Error &   &    3.48\% & 3.59\% & 0.34\% & 0.19\% & 3.66\% & 3.76\% & 0.45\% & 0.23\% \\
    $\mu_1=\mu_2$ &    36   & 26.24 & 30.26 & 93.34\% & 96.80\% & 0.13  & 0.15  & 63.96\% & 336.40 \\
    Error &       & 2.75\% & 3.73\% & 0.12\% & 0.06\% & 2.58\% & 3.57\% & 0.73\% & 0.86\% \\
    \midrule

    Simulation & 45 & 53.63 & 63.47 & 86.58\% & 93.39\% & 0.34  & 0.40  & 76.70\% & 341.16 \\
    $\mu_1\neq \mu_2$ & 45 & 54.84 & 65.37 & 86.08\% & 93.09\% & 0.34  & 0.41  & 76.33\% & 340.79 \\
    Error &       & 2.26\% & 3.00\% & 0.57\% & 0.32\% & 2.25\% & 2.99\% & 0.49\% & 0.11\% \\
    $\mu_1=\mu_2$ &    45   & 50.99 & 59.92 & 87.06\% & 93.67\% & 0.32  & 0.37  & 76.00\% & 336.40 \\
    Error &       & 4.93\% & 5.59\% & 0.56\% & 0.30\% & 4.94\% & 5.60\% & 0.92\% & 1.40\% \\
   	    \midrule
    Simulation & 60 & 112.64 & 138.75 & 71.55\% & 85.43\% & 0.94  & 1.15  & 90.58\% & 346.57 \\
    $\mu_1\neq \mu_2$ & 60 & 114.06 & 141.66 & 71.06\% & 85.03\% & 0.95  & 1.18  & 90.13\% & 346.46 \\
    Error &       & 1.26\% & 2.10\% & 0.68\% & 0.47\% & 1.47\% & 2.31\% & 0.50\% & 0.03\% \\
    $\mu_1=\mu_2$ &   60    & 104.76 & 127.56 & 73.42\% & 86.52\% & 0.87  & 1.06  & 89.67\% & 336.40 \\
    Error &       & 6.99\% & 8.06\% & 2.61\% & 1.28\% & 6.81\% & 7.88\% & 1.01\% & 2.94\% \\
    \midrule
    Simulation & 120 & 294.91 & 432.12 & 25.22\% & 54.23\% & 4.92  & 7.21  & 99.98\% & 377.06 \\
    $\mu_1\neq \mu_2$ & 120 & 293.92 & 434.13 & 25.42\% & 54.13\% & 4.90  & 7.24  & 99.96\% & 376.98 \\
    Error &       & 0.33\% & 0.46\% & 0.77\% & 0.18\% & 0.44\% & 0.36\% & 0.02\% & 0.02\% \\
    $\mu_1=\mu_2$ &   120    & 274.74 & 389.50 & 30.28\% & 58.85\% & 4.58  & 6.49  & 99.95\% & 336.40 \\
    Error &       & 6.84\% & 9.86\% & 20.06\% & 8.52\% & 6.94\% & 9.96\% & 0.03\% & 10.78\% \\

    \bottomrule
    \end{tabular}%
  \label{tab:comparison}%
\end{table}%

\begin{acknowledgement}
The authors would like to thank Marko Boon for carefully checking all calculations with {\em Mathematica}.
\end{acknowledgement}

\pagebreak

\end{document}